\documentclass[12pt,reqno]{amsart}
\usepackage{amsmath,amsfonts,amssymb,amsthm,amscd,latexsym, longtable}
\usepackage [latin1]{inputenc}

\usepackage{tikz}
\usetikzlibrary{arrows}
\usepackage{color}
\usepackage{multirow}

\usepackage{hyperref}					

\newtheorem{thm}{Theorem}[section]
\newtheorem{cor}[thm]{Corollary}

\newtheorem{prop}[thm]{Proposition}
\newtheorem{defn} [thm]{Definition}
\newtheorem{rem}[thm]{Remark}
\newtheorem{exam}[thm]{Example}

\newtheorem{prob}[thm]{Problem}

\numberwithin{equation}{section}

\begin{document}

\author[K.~K.~Kudaybergenov]{Karimbergen Kudaybergenov$^{1,2}$}
\address{$^1$Institute for Advances Study in Mathematics, Harbin Institute of Technologies,
Harbin, 150001, China}
\address{$^2$Suzhou Research Institute of Harbin Institute of Technology, Harbin Institute of Technologies, Suzhou, 215104, China}
    \email{\textcolor[rgb]{0.00,0.00,0.84}{kudaybergenovkk@gmail.com}}

\author[M.~M.~Ibragimov]{Mukhtar Ibragimov$^{3}$}

\address{$^3$Romanovskiy Institute of Mathematics,
    Uzbekistan Academy of Sciences,  Tashkent, 100174, Uzbekistan}
\email{\textcolor[rgb]{0.00,0.00,0.84}{m.ibragimov1909@gmail.com}}

\author[A.~D.~Arziev]{Allabay Arziev$^{3,4}$}
\address{$^4$Karakalpak State University,
     Nukus, 230112, Uzbekistan}

\email{\textcolor[rgb]{0.00,0.00,0.84}{allabayarziev@gmail.com}}



\title[On Completions and Dense Subspaces of Strongly Facially Symmetric Spaces]{On Completions and Dense Subspaces of Strongly Facially Symmetric Spaces}

\begin{abstract}
  We study the behavior of strongly facially symmetric spaces under completion and passage to norm-dense subspaces. We introduce a natural face-density condition guaranteeing that the completion of a normed SFS-space remains strongly facially symmetric. We prove that a norm-dense subspace of a neutral strongly facially symmetric space inherits the neutral SFS-structure whenever it is invariant under the ambient generalized Peirce projections. Several examples and counterexamples are presented, including intermediate subspaces of the trace class and the dense subspace $C[0,1]\subset L_1[0,1]$.
\end{abstract}

\keywords{Strongly facially symmetric space, norm-exposed face, face-density, generalized Peirce projections, dense subspace, trace class}

\maketitle

\bigskip

\section{Introduction}
Strongly facially symmetric (SFS) spaces were introduced and investigated by Friedman and Russo in \cite{FR1986,FR1989Affine} as a geometric framework for studying Banach spaces whose duals admit structures related to Jordan triple systems and operator algebras. The theory is based on the geometry of norm-exposed faces of the unit ball, together with the associated symmetries and generalized Peirce decompositions. Classical examples include preduals of $JBW^{\ast}$-triples (see \cite[Theorem~3.1]{FR1989Operator}), in particular $L_1$-spaces, as well as preduals of von Neumann algebras (see \cite[Theorem~2.11]{FR1989Operator}). The foundations of the theory were developed in the series of papers \cite{FR1986,FR1989Affine,FR1989Operator,FR1992,FR1993}, culminating in a detailed analysis of facial symmetries, generalized tripotents, and generalized Peirce projections.

A natural question, which does not seem to have been addressed explicitly in the literature, concerns the stability of the SFS property under completion and under passage to norm-dense (hereafter, for brevity, dense) subspaces. More precisely, if $Z$ is a normed SFS-space, does its completion remain strongly facially symmetric? Conversely, if $Z$ is a Banach SFS-space and $Y\subset Z$ is a norm-dense subspace, must $Y$ inherit the SFS-structure?

The first question is motivated by the observation that the fundamental objects of the theory---norm-exposed faces, facial symmetries, and generalized Peirce projections---are defined in terms of the norm topology and continuous linear functionals. Since surjective isometries and bounded projections extend naturally to completions, one may expect the SFS-structure to persist. We formulate a sufficient condition under which the completion of an SFS-space is again an SFS-space and discuss the role played by norm-exposed faces in this problem.

The second question turns out to be substantially more delicate. A dense subspace need not inherit the facial structure of the ambient Banach space. In particular, norm-exposed faces of the unit ball of the completion may disappear completely when restricted to the dense subspace. Moreover, facial symmetries and generalized Peirce projections arising in the completion need not preserve the dense subspace.

The main purpose of this note is to illustrate these phenomena. We show that the dense subspace
\[
C[0,1]\subset L_1[0,1],
\]
equipped with the $L_1$-norm, provides a natural obstruction. Although $L_1[0,1]$ is a classical neutral strongly facially symmetric space, we prove that $C[0,1]$ is not even weakly facially symmetric. More precisely, we exhibit a norm-exposed face of the unit ball of $C[0,1]$ which fails to be symmetric in the sense of Friedman and Russo. Consequently, the SFS property is not preserved under passage to norm-dense subspaces.

These observations lead naturally to the following general problem: characterize those dense subspaces of Banach SFS-spaces that inherit the facially symmetric structure. We hope that the examples presented here will be useful in understanding the extent to which the theory of facially symmetric spaces depends on completeness.

We introduce a natural face-density condition guaranteeing that the completion of an SFS-space remains strongly facially symmetric. We then prove that invariance under the ambient generalized Peirce projections is sufficient for a dense subspace to inherit the neutral SFS-structure.

The general theory is illustrated by several examples and counterexamples. Among them are the dense subspaces
\[
\mathcal F(H)\subset S_1(H),
\qquad
C[0,1]\subset L_1[0,1],
\]
and certain intermediate subspaces of the trace class. We also discuss the relation between complex and real SFS-spaces.

\section{Preliminaries}

In this section we briefly recall the basic notions of the theory of strongly facially symmetric spaces.

Let $Z$ be a real or complex normed space. Two elements $f,g\in Z$ are said to be \emph{orthogonal}, written
\[
f\diamond g,
\]
if
\[
\|f+g\|=\|f-g\|=\|f\|+\|g\|.
\]

A \emph{norm-exposed face} of the unit ball $Z_1$ is a nonempty set of the form
\[
F_u=\{f\in Z_1:u(f)=1\},
\]
where $u\in Z^{\ast}$ and $\|u\|=1$.

For a subset $S\subset Z$, define
\[
S^\diamond=\{f\in Z:f\diamond g\text{ for every }g\in S\},
\]
and call $S^\diamond$ the \emph{orthogonal complement} of $S$.

Subsets $S,T\subset Z$ are called \emph{orthogonal}, written $S\diamond T$, if $f\diamond g$ for all $(f,g)\in S\times T$.

An element $u\in Z^{\ast}$ is called a \emph{projective unit} if $\|u\|=1$ and
\[
\langle u,F_u^\diamond\rangle=0.
\]
We denote by $\mathcal{F}$ and $\mathcal{U}$ the sets of all norm-exposed faces of $Z_1$ and all projective units in $Z^{\ast}$, respectively.

A norm-exposed face $F$ is called \emph{symmetric} if there exists a surjective linear isometry
\[
S_F:Z\longrightarrow Z
\]
such that
\[
S_F^2=I
\]
and
\[
\operatorname{Fix}(S_F)=\overline{\operatorname{sp}(F)}\oplus F^\diamond.
\]
Such an isometry is called a \emph{symmetry corresponding to $F$}.

Associated with a symmetric face $F$ are the generalized Peirce projections
\[
P_2(F),\qquad P_1(F),\qquad P_0(F),
\]
defined by
\[
P_1(F)=\frac12(I-S_F),
\]
while
\[
P_2(F)+P_0(F)=\frac12(I+S_F),
\]
where $P_2(F)$ and $P_0(F)$ are the projections onto $\operatorname{sp}(F)$ and $F^\diamond$, respectively. These projections satisfy
\[
P_2(F)+P_1(F)+P_0(F)=I.
\]

A normed space $Z$ is called \emph{weakly facially symmetric} (WFS) if every norm-exposed face of $Z_1$ is symmetric.

A contractive projection $Q$ on a normed space $Z$ is called \emph{neutral} if
\[
\|Qx\|=\|x\|
\]
implies $Qx=x$. A WFS-space $Z$ is called neutral if, for every symmetric face $F$, the projection $P_2(F)$ is neutral.

A \emph{geometric tripotent} is a projective unit $u\in Z^{\ast}$ such that $F:=F_u$ is a symmetric face and
\[
S_F^{\ast}u=u
\]
for some symmetry $S_F$ corresponding to $F$. We denote by $\mathcal{SF}$ and $\mathcal{GU}$ the sets of all symmetric faces of $Z_1$ and all geometric tripotents in $Z^{\ast}$, respectively.

A WFS-space $Z$ is called \emph{strongly facially symmetric} (an SFS-space) if, for every symmetric face $F_u\in Z_1$ and every $v\in Z^{\ast}$ satisfying $\|v\|=1$ and $F_u\subset F_v$, one has
\[
S_u^{\ast}v=v,
\]
where $S_u$ is a symmetry corresponding to $F_u$.

\section{Completions of SFS-spaces}

The purpose of this section is to investigate when the completion of a normed strongly facially symmetric space remains strongly facially symmetric.

Let $Z$ be a real or complex normed space and let $\tilde{Z}$ denote its completion. We write
\[
Z_1=\{x\in Z:\|x\|\le 1\},
\qquad
\tilde{Z}_1=\{z\in\tilde{Z}:\|z\|\le 1\}.
\]
For $u\in Z^{\ast}$ with $\|u\|=1$, put
\[
F_u^ Z=\{x\in Z_1:u(x)=1\}.
\]
Since $Z$ is dense in $\tilde{Z}$, every $u\in Z^{\ast}$ has a unique norm-preserving extension $\widetilde u\in\tilde{Z}^\ast$. The corresponding norm-exposed face of $\tilde{Z}$ is
\[
F_{\widetilde u}^{\tilde{Z}}
 =\{z\in\tilde{Z}_1:\widetilde u(z)=1\}.
\]

Let $Z$ be a  strongly facially symmetric space. We say that $Z$ satisfies the \emph{face-density condition} (FD) if, for every $u\in Z^{\ast}$ with $\|u\|=1$ such that $F_{\widetilde u}^{\tilde{Z}}\ne\varnothing$, the face $F_u^Z$ is nonempty and
\[
\overline{F_u^Z}^{\,\tilde{Z}}=F_{\widetilde u}^{\tilde{Z}}.
\]

\begin{thm}\label{3.1}
Let $Z$ be a  strongly facially symmetric space satisfying \emph{(FD)}. Then its completion $\tilde{Z}$ is strongly facially symmetric.
\end{thm}

\begin{proof}
Let $\widetilde u\in\tilde{Z}^{\ast}$ be a norm-one functional and suppose that $F_{\widetilde u}^{\tilde{Z}}$ is a nonempty norm-exposed face of $\tilde{Z}_1$. By the canonical identification
\[
\tilde{Z}^{\ast}\cong Z^{\ast},
\]
there exists $u\in Z^{\ast}$ such that $\widetilde u$ is the unique extension of $u$. By assumption (FD),
\[
F_{\widetilde u}^{\tilde{Z}}=\overline{F_u^Z}^{\,\tilde{Z}}.
\]
Since $Z$ is an SFS-space, the face $F_u^Z$ is symmetric. Hence there exists a surjective linear isometry
\[
S_u:Z\longrightarrow Z,
\qquad
S_u^2=I_Z,
\]
associated with $F_u^Z$. In particular,
\begin{equation}\label{eq:fixSu}
\operatorname{Fix}(S_u)=\overline{\operatorname{sp}(F_u^Z)}\oplus(F_u^Z)^\diamond.
\end{equation}
Because $S_u$ is an isometry on the normed space $Z$, it is uniformly continuous. Therefore it extends uniquely to a surjective linear isometry
\[
\widetilde S_u:\tilde{Z}\longrightarrow\tilde{Z}
\]
satisfying
\[
\widetilde S_u|_Z=S_u.
\]
Moreover, since $S_u^2=I_Z$, by continuity we also have
\[
\widetilde S_u^2=I_{\tilde{Z}}.
\]
By (FD),
\[
\overline{F_u^Z}^{\,\tilde{Z}}=F_{\widetilde u}^{\tilde{Z}},
\]
and consequently
\[
\overline{\operatorname{sp}(F_u^Z)}^{\,\tilde{Z}}
 =\overline{\operatorname{sp}(F_{\widetilde u}^{\tilde{Z}})}^{\,\tilde{Z}}.
\]

We also show that
\[
(F_u^Z)^\diamond=(F_{\widetilde u}^{\tilde{Z}})^\diamond.
\]
Indeed, since $F_u^Z\subset F_{\widetilde u}^{\tilde{Z}}$, we have
\[
(F_u^Z)^\diamond\supset(F_{\widetilde u}^{\tilde{Z}})^\diamond.
\]
For the reverse inclusion, let $f\in(F_u^Z)^\diamond$ and take an arbitrary $g\in F_{\widetilde u}^{\tilde{Z}}$. By the definition of closure, there exists a sequence $(f_n)\subset F_u^Z$ such that $f_n\to g$ in $\tilde{Z}$. Since $f\diamond f_n$, we have
\[
\|f+f_n\|=\|f-f_n\|=\|f\|+\|f_n\|.
\]
By continuity of the norm,
\[
\|f+f_n\|\to\|f+g\|,
\qquad
\|f-f_n\|\to\|f-g\|,
\qquad
\|f_n\|\to\|g\|.
\]
Therefore
\[
\|f+g\|=\|f-g\|=\|f\|+\|g\|,
\]
so $f\diamond g$. Thus $f\in(F_{\widetilde u}^{\tilde{Z}})^\diamond$, which gives the reverse inclusion and hence the asserted equality.

It follows from \eqref{eq:fixSu} that
\[
\operatorname{Fix}(\widetilde S_u)
 =\overline{\operatorname{sp}(F_{\widetilde u}^{\tilde{Z}})}
  \oplus(F_{\widetilde u}^{\tilde{Z}})^\diamond.
\]
Therefore every nonempty norm-exposed face of $\tilde{Z}_1$ is symmetric. The generalized Peirce projections associated with $u$ on $Z$ are bounded linear operators and hence extend continuously to $\tilde{Z}$. In particular, the Peirce decomposition and the symmetry corresponding to the face pass from $Z$ to $\tilde{Z}$. Consequently, $\tilde{Z}$ is strongly facially symmetric.
\end{proof}

\begin{rem}
The condition (FD) is the essential technical point. Without it, the argument can fail: a functional may attain its norm on the completion $\tilde{Z}$ but fail to attain it on the normed space $Z$. Thus one should not suppress this point in a rigorous proof.
\end{rem}

The next example shows that the face-density condition (FD) is not restricted to classical sequence spaces. It also arises naturally in operator theory. Namely, the space of finite-rank operators is a norm-dense subspace of the Schatten trace class $S_1(H)$, and every norm-exposed face of the unit ball of $S_1(H)$ can be approximated by finite-rank elements belonging to the same face. Consequently, $\mathcal F(H)$ satisfies (FD), providing a nontrivial operator-theoretic illustration of the completion theorem.

\begin{exam}
Let $H$ be an infinite-dimensional Hilbert space and let
\[
Z=\mathcal F(H)
\]
be the space of all finite-rank operators on $H$, equipped with the trace norm
\[
\|x\|_1=\operatorname{Tr}(|x|).
\]
Then the completion of $Z$ is the Schatten trace class
\[
\tilde{Z}=S_1(H).
\]
We show that $Z$ satisfies the face-density condition (FD). Let $F$ be a nonempty norm-exposed face of the unit ball of $S_1(H)$. It is well known that such faces are of the form (see \cite[Lemma~2.3]{FR1989Operator})
\[
F=F_v=\{x\in S_1(H)_1:\operatorname{Tr}(xv^{\ast})=1\},
\]
where $v\in B(H)$ is a partial isometry. More explicitly,
\[
F_v=\{x\in S_1(H):x\ge0\text{ in the corner determined by }v,
\ \|x\|_1=1\},
\]
or equivalently,
\[
F_v=\{vh:h\in S_1(H)^+,\ s(h)\le v^{\ast}v,\ \operatorname{Tr}(h)=1\}.
\]
The corresponding face in $Z=\mathcal F(H)$ is
\[
F_v^Z=F_v\cap\mathcal F(H)
 =\{vh:h\in\mathcal F(H)^+,\ s(h)\le v^{\ast}v,\ \operatorname{Tr}(h)=1\}.
\]
Now take $x\in F_v$. Then $x=vh$ for some $h\in S_1(H)^+$ satisfying
\[
s(h)\le v^{\ast}v,
\qquad
\operatorname{Tr}(h)=1.
\]
Choose finite-rank positive operators $h_m\in\mathcal F(H)^+$ such that
\[
s(h_m)\le v^{\ast}v,
\qquad
\operatorname{Tr}(h_m)=1,
\]
and
\[
h_m\longrightarrow h\quad\text{in }S_1(H).
\]
For instance, take normalized spectral truncations of $h$. Put
\[
x_m=vh_m.
\]
Then $x_m\in\mathcal F(H)$ and $x_m\in F_v^Z$. Moreover,
\[
\|x_m-x\|_1=\|v(h_m-h)\|_1\le\|h_m-h\|_1\longrightarrow0.
\]
Therefore
\[
F_v\subset\overline{F_v^Z}^{\,S_1(H)}.
\]
The reverse inclusion is immediate because $F_v$ is norm closed. Hence
\[
\overline{F_v^Z}^{\,S_1(H)}=F_v.
\]
Thus the face-density condition (FD) holds for
\[
\mathcal F(H)\subset S_1(H).
\]
\end{exam}

\section{Dense subspaces of SFS-spaces}

The    preceding discussion suggests that the inheritance of the SFS-structure by dense subspaces is closely related to the behavior of facial symmetries and generalized Peirce projections. The following theorem shows that invariance under all generalized Peirce projections is sufficient for a dense subspace to inherit the neutral strongly facially symmetric structure.

\begin{thm}\label{4.1}
Let $Z$ be a Banach neutral strongly facially symmetric space, and let $Y\subset Z$ be a norm-dense linear subspace. Assume that, for every norm-exposed face $F$ of $Z_1$, the corresponding generalized Peirce projections
\[
P_0(F),\qquad P_1(F),\qquad P_2(F)
\]
satisfy
\[
P_k(F)Y\subset Y,
\qquad
k=0,1,2.
\]
Then $Y$, equipped with the inherited norm, is a neutral strongly facially symmetric space.
\end{thm}

\begin{proof}
Let $u\in Y^{\ast}$ with $\|u\|=1$, and suppose that
\[
F_u^Y=\{y\in Y_1:u(y)=1\}
\]
is nonempty. Since $Y$ is dense in $Z$, $u$ extends uniquely to a norm-one functional $\widetilde u\in Z^{\ast}$. Then
\[
F_u^Y=F_{\widetilde u}^Z\cap Y.
\]
Since $Z$ is an SFS-space, the face $F_{\widetilde u}^Z$ is symmetric. Let
\[
S_{\widetilde u}
 =P_2(\widetilde u)-P_1(\widetilde u)+P_0(\widetilde u)
\]
be its facial symmetry. By the assumed invariance of $Y$ under the generalized Peirce projections,
\[
S_{\widetilde u}(Y)\subset Y.
\]
Since $S_{\widetilde u}^2=I$, it follows that
\[
S_{\widetilde u}(Y)=Y.
\]
Hence
\[
S_u:=S_{\widetilde u}|_Y
\]
is a surjective linear isometry of $Y$ and $S_u^2=I_Y$. Moreover,
\[
\operatorname{Fix}(S_u)=Y\cap\operatorname{Fix}(S_{\widetilde u}).
\]
Since
\[
\operatorname{Fix}(S_{\widetilde u})
 =\overline{\operatorname{sp}(F_{\widetilde u}^Z)}\oplus(F_{\widetilde u}^Z)^\diamond,
\]
and $Y$ is invariant under the generalized Peirce projections, we obtain
\[
\operatorname{Fix}(S_u)=\overline{\operatorname{sp}(F_u^Y)}\oplus(F_u^Y)^\diamond.
\]
Thus $F_u^Y$ is symmetric.

Finally, neutrality also passes to $Y$. Indeed, if
\[
\|P_2(\widetilde u)y\|=\|y\|,
\qquad y\in Y,
\]
then neutrality in $Z$ gives
\[
P_2(\widetilde u)y=y.
\]
Therefore the  restricted projection $P_2(\widetilde u)|_Y$ is neutral. Hence $Y$ is a neutral SFS-space.
\end{proof}

\begin{rem}
Theorem~\ref{4.1} may be viewed as the main structural result of the paper. It identifies invariance under the ambient generalized Peirce projections as a sufficient condition for the inheritance of the neutral SFS-structure by dense subspaces.
\end{rem}

\begin{rem}
The theorem provides a general mechanism for constructing new neutral strongly facially symmetric spaces from known ones. In particular, it applies to dense subspaces of trace-class spaces and other preduals of $JBW^{\ast}$-triples whenever the corresponding generalized Peirce projections leave the subspace invariant.
\end{rem}

The preceding examples suggest that the inheritance of the SFS-property by subspaces is closely related to the behavior of facial symmetries and generalized Peirce projections. In the case of the trace class $S_1(H)$, these operators admit explicit descriptions and play a fundamental role in the facial structure of the unit ball. It is therefore natural to ask which intermediate subspaces
\[
\mathcal F(H)\subset Z\subset S_1(H)
\]
inherit the neutral strongly facially symmetric structure from $S_1(H)$.

The next result shows that invariance under all generalized Peirce projections is sufficient. In particular, every such invariant intermediate subspace is itself a neutral strongly facially symmetric space.

\begin{cor}\label{4.4}
Let $Z$ be a symmetric operator ideal satisfying
\[
\mathcal F(H)\subset Z\subset S_1(H).
\]
Then $Z$, equipped with the trace norm inherited from $S_1(H)$, is a neutral strongly facially symmetric space.
\end{cor}

\begin{proof}
Since $Z$ is a symmetric operator ideal,
\[
axb\in Z
\]
for every $x\in Z$ and all bounded operators $a,b\in B(H)$.
Let $u$ be a geometric tripotent of $S_1(H)$, and let
\[
l=l(u),
\qquad
r=r(u)
\]
be the corresponding support projections. The associated generalized Peirce projections are finite sums of maps of the form
\[
x\longmapsto lxr,
\qquad
x\longmapsto(1-l)x(1-r),
\]
and analogous corner operators (see \cite[Theorem~2.11]{FR1989Operator}). Since $Z$ is an operator ideal, all such operators leave $Z$ invariant. Therefore
\[
P_k(u)Z\subset Z,
\qquad
k=0,1,2.
\]
Since $\mathcal F(H)$ is dense in $S_1(H)$, the intermediate space $Z$ is norm dense in $S_1(H)$. Hence Theorem~\ref{3.1} applies and yields that $Z$ is a neutral strongly facially symmetric space.
\end{proof}

\begin{cor}\label{4.5}
Let $Z$ be a solid sequence space such that
\[
c_{00}\subset Z\subset\ell_1.
\]
Then $Z$, equipped with the $\ell_1$-norm, is a neutral strongly facially symmetric space.
\end{cor}

\begin{proof}
It is well known that $\ell_1$ is a neutral strongly facially symmetric space as the predual of $\ell_\infty$ (\cite[Theorem~2.11]{FR1989Operator}). Moreover, the generalized Peirce projections in $\ell_1$ are coordinate projections. More precisely, if $u=(u_n)\in\ell_\infty$ is a geometric tripotent and
\[
A=\{n\in\mathbb N:|u_n|=1\},
\]
then
\[
P_2(u)x=\chi_Ax,
\qquad
P_0(u)x=\chi_{\mathbb N\setminus A}x,
\qquad
P_1(u)=0.
\]
Since $Z$ is solid, for every $x\in Z$ and every $A\subset\mathbb N$ we have
\[
|\chi_Ax|\le|x|,
\]
and hence $\chi_Ax\in Z$. Thus $Z$ is invariant under all coordinate projections and therefore under all generalized Peirce projections of $\ell_1$.
Since $c_{00}\subset Z\subset\ell_1$, the subspace $Z$ is norm dense in $\ell_1$. Therefore, by Theorem~\ref{3.1}, $Z$ is a neutral strongly facially symmetric space.
\end{proof}

\begin{exam}
Let $H$ be an infinite-dimensional Hilbert space and set
\[
Z=S_1(H)\cap S_2(H),
\]
where $S_1(H)$ is the trace class and $S_2(H)$ is the Hilbert--Schmidt class. We equip $Z$ with the trace norm inherited from $S_1(H)$. Then
\[
\mathcal F(H)\subset Z\subset S_1(H).
\]
These inclusions are proper when $H$ is infinite-dimensional. Indeed,
\[
\mathcal F(H)\subsetneq S_1(H)\cap S_2(H)
\quad\text{and}\quad
S_1(H)\cap S_2(H)\subsetneq S_1(H).
\]
The space $Z$ is invariant under the generalized Peirce projections of $S_1(H)$. Indeed, these projections are given by left and right multiplication by projections and their finite sums, and such maps are contractive both on $S_1(H)$ and on $S_2(H)$. Therefore
\[
P_k(u)Z\subset Z,
\qquad
k=0,1,2.
\]
Consequently, by Corollary~\ref{4.4}, $Z$ is a neutral strongly facially symmetric space.
\end{exam}

The preceding theorem shows that invariance under generalized Peirce projections is a sufficient condition. The next example demonstrates that this condition is not automatic. We construct a dense intermediate subspace of $S_1(H)$ which is not invariant under the ambient generalized Peirce projections. Proposition~\ref{4.7} establishes only the failure of invariance. The failure of weak facial symmetry is treated separately in Proposition~\ref{4.9}.

\begin{prop}\label{4.7}
Let $H$ be an infinite-dimensional separable Hilbert space. Then there exists a complex linear subspace
\[
\mathcal F(H)\subset Z\subset S_1(H)
\]
which is norm dense in $S_1(H)$ but is not invariant under the  generalized Peirce projections of $S_1(H)$.
\end{prop}

\begin{proof}
Let $p\in B(H)$ be a projection such that both $pH$ and $(1-p)H$ are infinite-dimensional. Put
\[
q=1-p.
\]
Choose a positive trace-class operator $h\in S_1(H)$ satisfying
\[
h=php+qhq,
\]
where both $php$ and $qhq$ have infinite rank. Define
\[
Z=\mathcal F(H)+\mathbb C h.
\]
Then
\[
\mathcal F(H)\subset Z\subset S_1(H),
\]
and $Z$ is norm dense in $S_1(H)$, since $\mathcal F(H)$ is dense in $S_1(H)$.
Consider the projection $p$ as a geometric tripotent of $S_1(H)$. The corresponding Peirce-2 projection is
\[
P_2(p)x=pxp,
\qquad
x\in S_1(H).
\]
Since $h\in Z$, invariance of $Z$ under $P_2(p)$ would imply
\[
php=P_2(p)h\in Z.
\]
Suppose that $php\in Z$. Then there exist $f\in\mathcal F(H)$ and $\lambda\in\mathbb C$ such that
\[
php=f+\lambda h.
\]
Multiplying by $q$ on the left and right, we obtain
\[
0=qfq+\lambda qhq.
\]
Since $qfq$ has finite rank whereas $qhq$ has infinite rank, it follows that $\lambda=0$. Hence
\[
php=f\in\mathcal F(H),
\]
which contradicts the assumption that $php$ has infinite rank. Therefore
\[
php\notin Z.
\]
Consequently,
\[
P_2(p)Z\not\subset Z.
\]
Thus $Z$ is not invariant under the generalized Peirce projection $P_2(p)$ and hence is not invariant under the generalized Peirce projections of $S_1(H)$.
\end{proof}

\begin{rem}
Proposition~\ref{4.7} shows that norm-dense intermediate subspaces of a neutral strongly facially symmetric space need not be invariant under the ambient generalized Peirce projections. Thus the hypothesis of Theorem~\ref{4.1} is a genuine additional assumption.
\end{rem}

Whether every norm-dense neutral strongly facially symmetric subspace of a neutral SFS-space must be invariant under the ambient generalized Peirce projections remains an open problem.

The next example shows that lack of invariance under ambient generalized Peirce projections may indeed lead to failure of the SFS-property. More precisely, we exhibit a norm-dense subspace of $L_1[0,1]$ which is not weakly facially symmetric. The obstruction comes from the fact that the natural generalized Peirce projections of $L_1[0,1]$ do not preserve the subspace.

\begin{prop}\label{4.9}
Let $Z=C[0,1]$ be equipped with the $L_1$-norm
\[
\|f\|_1=\int_0^1|f(t)|\,dt.
\]
Then $Z$ is not weakly facially symmetric. More precisely, the norm-exposed face
\[
F=\left\{f\in Z_1:\int_0^{1/2}f(t)\,dt=1\right\}
\]
is not symmetric in the sense of Friedman--Russo.
\end{prop}

\begin{proof}
Put
\[
u=\chi_{[0,1/2]}\in L_\infty[0,1]=Z^{\ast}.
\]
Then $\|u\|=1$, and
\[
F=F_u=\{f\in Z_1:u(f)=1\}
\]
is a norm-exposed face.

The face $F$ is nonempty. Indeed, take any nonnegative continuous function $f$ with
\[
\operatorname{supp} f\subset(0,1/2)
\qquad\text{and}\qquad
\int_0^{1/2}f(t)\,dt=1.
\]
Then $f\in F$.

Now let $f\in F$. Since
\[
1=\int_0^{1/2}f(t)\,dt
 \le\int_0^{1/2}|f(t)|\,dt
 \le\int_0^1|f(t)|\,dt
 \le1,
\]
all inequalities are equalities. Hence
\[
f(t)\ge0\quad\text{a.e. on }[0,1/2]
\]
and
\[
f(t)=0\quad\text{a.e. on }(1/2,1].
\]
Since $f$ is continuous, this implies
\[
f(t)=0,
\qquad
t\in[1/2,1].
\]
Thus
\[
\overline{\operatorname{sp} F}=\{f\in C[0,1]:f(t)=0\text{ for }t\in[1/2,1]\}.
\]
On the other hand, orthogonality in the $L_1$-norm is disjointness of supports. Therefore
\[
F^\diamond=\{g\in C[0,1]:g(t)=0\text{ for }t\in[0,1/2]\}.
\]
Consequently,
\[
\overline{\operatorname{sp} F}\oplus F^\diamond
 =\{h\in C[0,1]:h(1/2)=0\}.
\]
Assume, for contradiction, that $F$ is symmetric. Then there exists a surjective linear isometry $S:Z\to Z$ such that
\[
S^2=I
\]
and
\[
\operatorname{Fix}(S)=\overline{\operatorname{sp} F}\oplus F^\diamond
 =\{h\in C[0,1]:h(1/2)=0\}.
\]
Since $S$ is an isometry for the $L_1$-norm, it extends uniquely to a surjective linear isometry
\[
\widetilde S:L_1[0,1]\longrightarrow L_1[0,1].
\]
But the subspace
\[
\{h\in C[0,1]:h(1/2)=0\}
\]
is dense in $L_1[0,1]$, because changing a function at one point is invisible in $L_1$. Hence $\widetilde S$ fixes a dense subspace of $L_1[0,1]$. By continuity,
\[
\widetilde S=I.
\]
Therefore $S=I$ on $C[0,1]$, and hence
\[
\operatorname{Fix}(S)=C[0,1],
\]
which contradicts
\[
\operatorname{Fix}(S)=\{h\in C[0,1]:h(1/2)=0\}.
\]
Thus $F$ is not symmetric.
\end{proof}

\begin{rem}
The failure of weak facial symmetry is caused by the fact that the natural generalized Peirce projections in $L_1[0,1]$ associated with the face $F$ do not preserve the subspace $C[0,1]$.
\end{rem}

\section{Final remarks and open problems}

The results of this paper suggest that invariance under generalized Peirce projections plays a fundamental role in the inheritance of the SFS-structure.

\begin{prob}
Characterize those norm-dense subspaces of a neutral strongly facially symmetric space which inherit the neutral SFS-structure. In particular, is invariance under the ambient generalized Peirce projections also necessary?
\end{prob}

To formulate the next problem, we recall several notions from the theory of facially symmetric spaces introduced by Friedman and Russo (see \cite{FR1992} for details).

\begin{defn}
A normed space $Z$ is called \emph{atomic} if every symmetric face of the unit ball $Z_1$ contains an extreme point.
\end{defn}

We denote by $\operatorname{ext} Z_1$ the set of all extreme points of $Z_1$. A point $x\in Z_1$ is called a \emph{norm-exposed point} of $Z_1$ if
\[
\{x\}=Z_1\cap H
\]
for some hyperplane $H$. We denote by $\operatorname{exp} Z_1$ the set of all norm-exposed points of $Z_1$.

\begin{defn}
Let $Z$ be a neutral SFS-space. We say that $Z$ satisfies the zxiom:
\begin{itemize}
\item[--]  \textup{(PE)} if
\[
\operatorname{exp} Z_1=\operatorname{ext} Z_1;
\]
\item[--]   \textup{(STP)} if, for every pair $f,g\in\operatorname{ext} Z_1$,
\[
\overline{\langle f\mid g\rangle}=\langle g\mid f\rangle,
\]
where the bar denotes complex conjugation;
\item[--]   \textup{(ERP)} if, for every geometric tripotent $u$ and every $f\in\operatorname{ext} Z_1$, the element $P_2(u)f$ is a scalar multiple of some element of $\operatorname{ext} Z_1$;
\item[--]   \textup{(JP)} if, for every pair of mutually orthogonal geometric tripotents $u$ and $v$,
\[
S_uS_v=S_{u+v};
\]
\item[--] \textup{(FE)} if every norm-closed face of the unit ball $Z_1$ is norm exposed.
\end{itemize}
\end{defn}

\begin{prob}
Let $Z$ be a strongly facially symmetric space satisfying the face-density condition \textup{(FD)}. If $Z$ is neutral (respectively, atomic), is its completion $\tilde{Z}$ neutral (respectively, atomic)? If $Z$ satisfies one of the axioms \textup{(PE)}, \textup{(STP)}, \textup{(ERP)}, \textup{(JP)}, or \textup{(FE)}, does $\tilde{Z}$ inherit the same property?
\end{prob}

\end{document}